\documentclass[10pt, reqno, a4paper]{amsart}

\usepackage[latin1]{inputenc}
\usepackage{amsfonts}
\usepackage{amsmath}
\usepackage{amssymb}
\usepackage{amsthm}
\usepackage[T1]{fontenc}
\usepackage{enumerate}
\usepackage{verbatim}
\usepackage{lmodern}
\usepackage{tensor}

\RequirePackage[dvipsnames,usenames]{color}  
\definecolor{VeryDarkGreen}{rgb}{0,0.18,0.08}
\definecolor{VeryDarkBrown}{rgb}{0.12,0.08,0.04}
\usepackage[pdftex]{hyperref}
\hypersetup{
bookmarks,
bookmarksdepth=2,
bookmarksopen,
bookmarksnumbered,
pdfstartview=FitH,
colorlinks,backref,hyperindex,
linkcolor=VeryDarkBrown,
citecolor=VeryDarkGreen,
urlcolor=VeryDarkGreen
}
\usepackage[matrix, arrow]{xy}
\usepackage{listings}

\title[Bernstein-Sato polynomials in positive characteristic]{Bernstein-Sato polynomials and test modules in positive characteristic}
\author{Manuel Blickle}
\author{Axel St\"abler}

\DeclareMathOperator{\Supp}{Supp}

\DeclareMathOperator{\Hom}{Hom}

\DeclareMathOperator{\id}{id}

\DeclareMathOperator{\colim}{colim}
\DeclareMathOperator{\End}{End}

\DeclareMathOperator{\Spec}{Spec}

\newcommand{\eps}{\varepsilon}
\input xy
\xyoption{all}
\SelectTips{cm}{10}
\begin{document}
\swapnumbers
\theoremstyle{plain}
\newtheorem{Le}{Lemma}[section]
\newtheorem{Ko}[Le]{Corollary}
\newtheorem{Theo}[Le]{Theorem}
\newtheorem*{TheoB}{Theorem}
\newtheorem{Prop}[Le]{Proposition}
\newtheorem*{PropB}{Proposition}
\newtheorem{Con}[Le]{Conjecture}
\theoremstyle{definition}
\newtheorem{Def}[Le]{Definition}
\newtheorem*{DefB}{Definition}
\newtheorem{Bem}[Le]{Remark}
\newtheorem{Bsp}[Le]{Example}
\newtheorem{Be}[Le]{Observation}
\newtheorem{Sit}[Le]{Situation}
\newtheorem{Que}[Le]{Question}
\newtheorem{Dis}[Le]{Discussion}
\newtheorem{Prob}[Le]{Problem}
\newtheorem*{Konv}{Conventions}

\def\cocoa{{\hbox{\rm C\kern-.13em o\kern-.07em C\kern-.13em o\kern-.15em
A}}}

\address{Manuel Blickle\\
Johannes Gutenberg-Universit\"at Mainz\\ Fachbereich 08\\
Staudingerweg 9\\
55099 Mainz\\
Germany}
\email{blicklem@uni-mainz.de}

\address{Axel St\"abler\\
Johannes Gutenberg-Universit\"at Mainz\\ Fachbereich 08\\
Staudingerweg 9\\
55099 Mainz\\
Germany}
\email{staebler@uni-mainz.de}

\date{\today}

\begin{abstract}
In analogy with the complex analytic case, Musta\c{t}\u{a} constructed (a family of) Bernstein-Sato polynomials for the structure sheaf $\mathcal{O}_X$ and a hypersurface $(f=0)$ in $X$, where $X$ is a regular variety over an $F$-finite field of positive characteristic (see \cite{mustatabernsteinsatopolynomialspositivechar}). He shows that the suitably interpreted zeros of his Bernstein-Sato polynomials correspond to the $F$-jumping numbers of the test ideal filtration $\tau(X,f^t)$. In the present paper we generalize  Musta\c{t}\u{a}'s construction replacing $\mathcal{O}_X$ by an arbitrary $F$-regular Cartier module $M$ on $X$ and show an analogous correspondence of the zeros of our Bernstein-Sato polynomials with the jumping numbers of the associated filtration of test modules $\tau(M,f^t)$ provided that $f$ is a non zero-divisor on $M$.
\end{abstract}

\maketitle

\section*{Introduction}

To keep notation simple in this introduction let $X = \Spec k[x_1,\ldots,x_n]$ be the affine $n$-space over an algebraically closed field $k$. Denote the polynomial ring by $R=k[x_1,\ldots,x_n]$ and fix an equation $f \in R$ defining a hypersurface in $X$.
We denote by $\gamma \colon \Spec R \to \Spec R[t]$ the graph embedding of $f$ given by sending $t$ to $f$.

If $k = \mathbb{C}$ one has the \emph{Bernstein-Sato Polynomial} of $f$ which is an important measure of the singularities of the hypersurface defined by $f=0$. It is defined to be the non-zero monic polynomial of minimal degree among those $b(s) \in k[s]$ such that \[ b(s)f^s=Pf^{s+1} \] for some differential operator $P \in \mathcal{D}_R[s]=k[x_1,\ldots,x_n,\partial_{x_1},\ldots,\partial_{x_n}][s]$.

Kashiwara and Malgrange interpret in \cite{kashiwaravfiltration} and \cite{malgrangevanishingdmodule} the Bernstein-Sato polynomial as the minimal polynomial of the action of the Euler operator $ \frac{\partial}{\partial t}t$ on a graded piece of the $V$-filtration of the $\mathcal{D}_R$-module pushforward $\gamma_+R$ along the graph embedding. In fact, a key point in the work of Kashiwara and Malgrange is the construction of said $V$-filtration in a much more general context, namely for regular holonomic $\mathcal{D}_R$-modules, which they achieve by their theory of $b$-functions, which generalizes the Bernstein-Sato polynomial for a hypersurface equation recalled above. By work of Budur and Saito \cite{budursaitomultiplieridealvfiltration} from the $V$-filtration on the $\mathcal{D}_{R[t]}$-module $\gamma_+R$ one can reconstruct the filtration of multiplier ideals $\mathcal{J}(R,f^t) \subseteq R$ for $0<t\leq 1$. This shows, in particular, that the jumping numbers of the multiplier ideal filtration between $0$ and $1$ are  zeros of the Bernstein-Sato polynomial.

A consequence of the existence of the Bernstein-Sato Polynomial is that the $\mathcal{D}_R$-module $R_f$ is generated by $1/f$  if (and only if) the reduced Bernstein-Sato Polynomial $(x+1)^{-1}b(s)$ does not have negative integral roots \cite{torrelli}. However, if $k$ is a field of positive characteristic $p>0$, then it is shown in \cite{alvarezblicklelybeznik} that the $\mathcal{D}_R$-module $R_f$ is \emph{always} generated by $1/f$. Hence, there cannot be a theory of Bernstein-Sato polynomials in positive characteristic with the same defining property. This observation is just one example for the fact that $\mathcal{D}$-module theory in positive characteristic is quite different from the complex case.
 
However, by taking the interpretation of the Bernstein-Sato polynomial as the minimal polynomial of an action of the Euler operator (due to  Kashiwara and Malgrange) as his point of departure, Musta\c{t}\u{a} defines in \cite{mustatabernsteinsatopolynomialspositivechar} a family of Bernstein-Sato polynomials for a hypersurface $f=0$ over a field of positive characteristic. Contrary to the complex analytic case it is not enough to consider the action of the Euler operator alone, instead one has to also consider all higher divided power Euler operators $\vartheta_i = \partial_t^{[p^i]} t^{p^i}$ at once.\footnote{Note that the order is reversed here. That is, one usually considers $t^{p^i} \partial^{[p^i]}_t$. We will be able to use this standard convention once we switch to right modules.}

More precisely, for $e \geq 1$ let $M^e_f$ be the $\mathcal{D}^e_R[\vartheta_1, \ldots, \vartheta_{p^{e-1}}]$-module generated by the image of $\gamma_\ast R$ in $\gamma_+ R$, where $\mathcal{D}^e_R$ is the subring consisting of those differential operators which are linear over $R^{p^e}$. The Euler operators $\vartheta_i$ act on the quotient $M^e_f/t M^e_f$ for $1 \leq i \leq e-1$ with eigenvalues in $\mathbb{F}_p$. The $e$th \emph{Bernstein-Sato polynomial} as introduced by Musta\c{t}\u{a} encodes the common eigenvalues of these operators. Musta\c{t}\u{a} proved furthermore that the information of these eigenvalues (suitably lifted to $\mathbb{Q}$) is equivalent to the data of the $F$-jumping numbers of the test ideal filtration $\tau(R,f^t)$ of $f$ in the range $(0,1]$. As the test ideal can be viewed as a positive characteristic analog of the multiplier ideal this statement is a characteristic $p$ version of the result of Budur and Saito that the jumping numbers of the multiplier ideal are zeroes of the classical Bernstein-Sato polynomial as alluded to above.

Work of the second named author in \cite{staeblertestmodulnvilftrierung} suggests that in positive characteristic the test module filtration itself is a suitable analog of the $V$-filtration: For one thing there is a certain axiomatic characterization of the test module filtration similar to that of the $V$-filtration but also different in the sense that the action of the differential operators is replaced by a (right) action of the Frobenius. Furthermore, a certain associated graded piece of the test module filtration corresponds, via an analogue of the Riemann-Hilbert correspondence, to a functor on perverse constructible sheaves of $\mathbb{F}_p$-vector spaces that has several of the desirable properties of nearby cycles in the $\ell \neq p$-case. This relationship between nearby cycles and $\mathcal{D}$-modules in characteristic $0$ was the motivation behind the construction of the $V$-filtration for holonomic $\mathcal{D}$-modules as a way to realize the nearby cycles functor for constructible $\mathbb{C}$-sheaves on the $\mathcal{D}$-module side.

What we achieve in the present paper is to also generalize Musta\c{t}\u{a}'s theory of Bernstein-Sato polynomials to this more general context where the test module filtration is defined and well behaved as in \cite{staeblertestmodulnvilftrierung}. In order to state our results let us recall some background on Cartier modules and their test modules from \cite{blicklep-etestideale}.

Let us from now on assume that $R$ is an $F$-finite Noetherian ring of positive characteristic $p$. A Cartier module $M$ (over $R$) is an $R$-module together with an $R$-linear map $\kappa: F_\ast M \to M$, where $F: R \to R$ is the absolute Frobenius given by $x \mapsto x^p$. A Cartier submodule of $M$ is an $R$-submodule $N$ such that $\kappa(N) \subseteq N$. We say that $M$ is \emph{$F$-pure} if $\kappa$ is surjective. We call $M$ \emph{$F$-regular} if $M$ is $F$-pure and if for any Cartier submodule $N$ of $M$ which after localizing at every generic point of $\Supp M$ agrees with $M$ we have $N = M$.

Let $M$ be an $F$-regular Cartier module which as an $R$-module is finitely generated. Let $f$ be a non zero-divisor on $M$. Then the \emph{test module} with respect to the ideal $(f) \subseteq R$ and $t \in \mathbb{R}_{\geq 0}$ is
\[
\tau(M, f^t) = \sum_{e \geq 1} \kappa^e f^{\lceil tp^e\rceil} f \underline{M}_\mathcal{C},
\]
where $\mathcal{C}$ is the algebra generated for $n \geq 1$ by the $\kappa^n f^{\lceil tp^n\rceil}$ and $\underline{M}_\mathcal{C} = (\mathcal{C}_+)^h M$ for all $h \gg 0$ is the stable image (cf. \cite[Proposition 2.13]{blicklep-etestideale}). It follows from \cite[Theorem 3.11]{blicklep-etestideale}, \cite[Lemma 3.1]{staeblertestmodulnvilftrierung} that the definition we give here is in fact equivalent to the definition of test modules in \cite{blicklep-etestideale}. Moreover, we will simplify the description of the test module in Section \ref{SectionTestDmodules}.

Test modules form a decreasing filtration of $R$-submodules of $M$, i.\,e.\ $\tau(M, f^s) \subseteq \tau(M, f^t)$ for $s \geq t$. This filtration is \emph{right continuous}, that is, for $\eps \ll 1$ one has $\tau(M, f^t) = \tau(M,f^{t + \eps})$. An element $t \in \mathbb{Q}$ such that for all $\eps > 0$ one has $\tau(M, f^t) \neq \tau(M, f^{t -\eps})$ is called an \emph{$F$-jumping number (of the test module filtration along $f$)}. Test module filtrations satisfy the so-called Brian\c{c}on-Skoda theorem, namely for any $t \geq 1$ one has $\tau(M, f^{t}) = f \tau(M, f^{t-1})$. In particular, it suffices to control the $F$-jumping numbers in the range $(0,1]$. Moreover, if $R$ is essentially of finite type over an $F$-finite field then the set of $F$-jumping numbers in $(0,1]$ is finite (\cite[Corollary 4.19]{blicklep-etestideale}) and all $F$-jumping numbers are rational (the rationality is a formal consequence of the finiteness similar to the argument in \cite[Theorem 3.1]{blicklemustatasmithdiscretenessrationality}).

Similar to Musta\c{t}\u{a}'s approach in \cite{mustatabernsteinsatopolynomialspositivechar} we use the graph embedding along the fixed hypersurface $f=0$ to define a family of Bernstein-Sato polynomials $b^e_{f,M}(s) \in \mathbb{Q}[s]$. This willl be done by exploiting a system of \emph{right} $\mathcal{D}^e_R$-modules which arises from the Cartier module structure of $M$. This will be explained in the following sections. Our main result can now be stated as follows:

\begin{TheoB}[Theorem \ref{MainResult}]
Let $R$ be regular essentially of finite type over an $F$-finite field. Let $(M, \kappa)$ be an $F$-regular Cartier module and $f \in R$ a non zero-divisor on $M$. The roots of the Bernstein-Sato polynomials $b^e_{M,f}(s)$ are given for $e$ sufficiently large by $\frac{\lceil \lambda p^e \rceil -1}{p^e}$, where $\lambda$ varies over the $F$-jumping numbers of the test module filtration $\tau(M, f^t)$ for $t \in (0,1]$.
\end{TheoB}
The crucial point is that the a priori infinite collection of Bernstein-Sato polynomials $b^e_{M,f}(s)$ for $e \geq 0$ is completely determined by a finite collection of rational numbers, namely the jumping numbers of the test module filtration attached to $(M,f)$.

In conclusion we would like to draw the reader's attention to the recent work of Stadnik \cite{stadnikbernsteinsatopoly} who also addresses the problem of extending Musta\c{t}\u{a}'s Bernstein-Sato polynomials to a more general context. Stadnik, however, works in the context of Emerton and Kisin's category of unit $R[F]$-modules \cite{emertonkisinrhunitfcrys}. In order to prove his existence result for $b$-functions he essentially has to reconstruct a theory of test ideals in this context, which he coins \emph{list-test-ideals} in \cite{stadnikbernsteinsatopoly}. It was one motivation of the authors of the present paper to point out that, by working in the essentially equivalent theory of Cartier modules (more precisely Cartier crystals, see \cite[Section 5.2]{blickleboecklecartierfiniteness}) one can rely on the already existing theory of test modules. By further replacing left $\mathcal{D}_R$-modules by right $\mathcal{D}_R$-modules there is a natural construction of the Bernstein-Sato polynomials with the desired link to the $F$-jumping numbers of the test module filtration. We show in Section \ref{stadniksbfunctions} that Stadnik's $b$-functions are precisely the limit over our Bernstein-Sato polynomials $b^e_{M,f}(s)$.

\subsection*{Acknowledgements}
Both authors were supported by SFB/Transregio 45 Bonn-Essen-Mainz financed by Deutsche Forschungsgemeinschaft. The idea for this paper was conceived while the second named author was a guest of the University of Michigan. In particular, he would like to thank Mircea Musta\c{t}\u{a} and Karen Smith for making this visit possible.

\section{$\mathcal{D}$-modules in positive characteristic}

Throughout this article we assume all rings to contain a field of prime characteristic $p > 0$.
The absolute Frobenius homomorphism given by sending $r \mapsto r^p$ is denoted by $F: R \to F_\ast R$. For an $R$-module $M$ we denote by $F^e_\ast M$ the $R$-module whose underlying abelian group is $M$ but with multiplication given by $r \cdot m = r^{p^e} m$. The ring $R$ is called $F$-finite if $F_\ast R$ is a finite $R$-module; in other words, the Frobenius morphism on $\Spec R$ is a finite map.

Given a ring $R$ we denote by $\mathcal{D}_R$ the ring of (absolute) $\mathbb{Z}$-linear differential operators in the sense of Grothendieck \cite{EGAIV-4}. Given a polynomial ring $R[t]$, we write $\partial_t^{[m]} \colon R[t] \to R[t]$ for the $R$-linear differential operator which sends $t^n$ to $\binom{n}{m} t^{n-m}$ with the usual convention that $\binom{n}{m} = 0$ for $m > n$. We introduce the notation
\[
    \theta_m = t^m \partial_t^{[m]} \qquad \qquad \vartheta_m = \partial_t^{[m]}t^m
\]
for the $R$-linear operators which are given by sending $t^n \mapsto \binom{n}{m}t^n$ and $t^n \mapsto \binom{n+m}{m}t^n$ respectively. The operator $\theta_m$ is called the \emph{(divided power) Euler operator of order $m$}.

As $R$ is a ring of prime characteristic $p>0$ one has the $p$-filtration of its ring of differential operators, see \cite{chasepfiltration}. The ring of (absolute) differential operators $\mathcal{D}_R$ is the direct limit of rings
\[
    \mathcal{D}^e_R \cong \End_R(F^e_\ast R),
\]
called the differential operators of level $e$. Indeed, the inclusion from $\mathcal{D}^e_R \to \mathcal{D}^{e+1}_R$ is the composition of the natural map $\End_R(F_\ast^e R) \to F_\ast \End_R(F_\ast^e R)$ followed by $F_\ast \End_R(F_\ast^e R) \to \End_R(F_\ast^{e+1} R)$. The direct limit over these maps yields $\mathcal{D}_R$.

If one uses $R^p \subseteq R$ instead of $R \to F_\ast R$ then one obtains the more familiar but equivalent description $\mathcal{D}^e_R = \End_{R^{p^e}}(R)$ and the union over these is $\mathcal{D}_R$. Also note that if $R[t]$ is a polynomial ring over $R$ then we have an inclusion $\mathcal{D}_R \to \mathcal{D}_{R[t]}$. Indeed, $F_\ast R[t] = \bigoplus_{i=0}^{p^e -1} F_\ast R t^i$ and given $P \in \mathcal{D}^e_R$ we have an extension $P'$ by sending $b t^i$ to $P(b) t^i$ for $i \geq 0$.

We will denote by Mod-$\mathcal{D}^e_R$ the category of right $\mathcal{D}^e_R$-modules and by Mod-$R$ the category of right $R$-modules.

We recall a theorem of Lucas (\cite[Section XXI]{lucasbinom}) which states that given natural numbers $n, m$ with $p$-adic expansions $n = \sum_{i=0}^s a_i p^i$ and $m = \sum_{i=0}^s b_i p^i$ with $a_i, b_i \in \{0, \ldots, p-1\}$ one has \[ \binom{n}{m} = \prod_{i=0}^s \binom{a_i}{b_i}\quad \text{mod } p.\]
In particular, it is a crucial ingredient in some proofs of the following relations among the differential operators in positive characteristic which we recall for the convenience of the reader.
\begin{Le}
\label{DiffOpFormulae}
Let $R$ be a regular and $F$-finite ring and $R[t]$ the polynomial ring over $R$ in one variable. Then the following hold:
\begin{enumerate}[(a)]
 \item $[\partial_t^{[p^i]}, t^{p^i}] = 1$ which just means $\vartheta_i = 1 + \theta_i$.
\item$\frac{(sr)!}{(s!)^r} \partial_t^{[sr]} = (\partial_t^{[s]})^r$.
\item $\prod_{j=1}^r(\theta_{p^e} + j) = (\partial_t^{[p^e]})^r (t^{p^e})^r$.
\item $[t, \theta_{p^i}] =  -\theta_{p^i -1} t -t$ for all $i$.
\item$[\theta_i, \theta_j] = 0$ for all $i,j$.
\item $[t, \theta_m] = -mt - t \sum_{j=0}^{m-1} \theta_j$.
\item $\theta_m \in \mathcal{D}^e_R[\theta_1, \theta_p, \ldots, \theta_{p^{e-1}}]$ for all $m < p^e$.
\item $\mathcal{D}^e_R[\theta_1, \ldots, \theta_{p^{e-1}}]t = t \mathcal{D}^e_R[\theta_1, \ldots, \theta_{p^{e-1}}]$.
\end{enumerate}
\end{Le}
\begin{proof}
(a) and (b) are proven in \cite[Lemma 4.1]{mustatabernsteinsatopolynomialspositivechar}, and (c), (d), (e), (f) follow from (a) and loc.\ cit. (g) follows from \cite[Remark 6.3]{mustatabernsteinsatopolynomialspositivechar} and (a). For (h) we argue along the lines of \cite[Lemma 6.4]{mustatabernsteinsatopolynomialspositivechar}. The inclusion from right to left follows from (d) and (g). The other inclusion follows similarly using (f).
\end{proof}

\section{From Cartier modules to right $\mathcal{D}$-modules}
\label{SectionPreliminaries}

In this section we recall the construction of the functor from Cartier modules to (right) $\mathcal{D}_R$-modules.

Throughout we assume that $R$ is an $F$-finite regular ring and that all modules considered are finitely generated.
By definition, a Cartier module $M$ is an $R$-module $M$ together with an $R$-linear map $\kappa: F_\ast M \to M$. This is equivalent to the data of an $R$-module and an $R$-linear map $C: M \to F^! M = \Hom_R(F_\ast R, M)$ ($C$ is just the adjoint of $\kappa$ -- see \cite[Proposition 2.18]{blickleboecklecartierfiniteness}). Iterating this map we obtain a directed system $M \to F^!M \to \ldots \to {F^e}^! M$. We have

\begin{Prop}
\label{CartierToRightDMod}
Let $(M, \kappa)$ be a Cartier module. Then the limit $\mathcal{M}$ over the maps $M \to {F^e}^!M$ yields an isomorphism $\mathcal{M} \to F^! \mathcal{M}$ which endows $\mathcal{M}$ with a right $\mathcal{D}_R$-module structure.
\end{Prop}
\begin{proof}
It is easy to see that the $C^e$ induce a map $\mathcal{M} \to {F^e}^! \mathcal{M}$ which is an isomorphism for all $e \geq 0$. Each ${F^e}^! M$ is naturally a right $\mathcal{D}^e_R$-module by premultiplication. This induces a right $\mathcal{D}_R$-module structure in the limit.
\end{proof}

It is well known that if $R$ is smooth over a perfect field $k$ then the top-dimensional differential forms $\omega_{R/k}$ are naturally equipped with a right $\mathcal{D}_R$-module structure and $\omega_{R/k}$ induces an equivalence between left and right $\mathcal{D}_R$-modules (see \cite[Chapitre 1]{berthelotdmodulesarithmetiqueii}).

If $k$ is only $F$-finite but not perfect then the situation is more complicated. We proceed as follows. Fix, once and for all, an isomorphism $k \to F^! k$. If $R$ is regular essentially of finite type over $k$ with structural morphism $f: \Spec R \to k$ then we set $\omega_R := f^! k$ and we get an induced isomorphism $\omega_R \to F^! \omega_R$ (note that $F$ is the \emph{absolute} Frobenius morphism)\footnote{We are supressing a shift here, but $f^! k$ is supported in a single degree and is an invertible sheaf.}. This isomorphism endows $\omega_R$ with a right $\mathcal{D}_R$-module structure and after this choice one has an equivalence between right and left $\mathcal{D}_R$-modules that is obtained by tensoring with $\omega_R^{-1}$.

In particular, since direct limits commute with tensor products the category of $\mathcal{D}_R$-modules obtained in Proposition \ref{CartierToRightDMod} (together with the fixed isomorphism) is equivalent to the category of unit $R[F]$-modules of Emerton and Kisin (see \cite[Theorem 2.27]{blicklegammasheaves}). Moreover, if we restrict the functor $M \mapsto \colim_e {F^e}^! M$ to the category of minimal Cartier modules (or equivalently, if we descend it to Cartier crystals -- see \cite{blickleboecklecartierfiniteness} and \cite{blickleboecklecartierekequivalence} for these notions) then it is also fully faithful.

Since test modules are naturally attached to Cartier modules it seems more natural to work with \emph{right}-$\mathcal{D}_R$ modules when studying test module filtrations and Bernstein-Sato polynomials. In fact, using this approach we can employ the ordinary higher Euler operators (i.\,e.\ $t^{p^i} \partial_t^{[p^i]}$ instead of $\partial_t^{[p^i]} t^{p^i}$) and avoid a sign change in the definition of Bernstein-Sato polynomials compared to \cite{mustatabernsteinsatopolynomialspositivechar}.

Note that the case Musta\c t\u a considered in \cite{mustatabernsteinsatopolynomialspositivechar} corresponds in our setting to the Cartier module $\omega_R$ with the (chosen) isomorphism $\omega_R \to F^!\omega_R$. In terms of constructible sheaves on the \'etale site this corresponds, via the Riemann-Hilbert correspondence of \cite{emertonkisinrhunitfcrys} and \cite[Theorem 5.15]{blickleboecklecartierfiniteness}, to the case of the constant sheaf.

For left $\mathcal{D}^e_R$-modules one has, as a special case of Morita equivalence, an equivalence with $R$-Mod for all $e \geq 0$ (see e.\,g.\ \cite[Proposition 3.8 and Corollary 3.10]{blicklediss}). A similar statement holds for right modules:

\begin{Prop}
\label{EquivalenceModRModD}
Let $R$ be an $F$-finite regular ring. Then the functor ${F^e}^!(-) = \Hom_R(F^e_\ast R, -)$ induces an equivalence between Mod-$R$ and Mod-$\mathcal{D}^e_R$. Its inverse is given by $- \otimes_{\End_R(F_\ast^e R)} F_\ast^e R$. In particular, ${F^e}^!$ reflects isomorphisms.
\end{Prop}
\begin{proof}
By the assumptions on $R$ we obtain that $F_\ast^e R$ is a finitely generated locally free $R$-module. Note that ${F^e}^! M = \Hom(F_\ast^e R, M)$ viewed as an $R$-module via the ring isomorphism $R \to F_\ast^e R$ and acting on the first factor is isomorphic to ${F^e}^\ast M \otimes_{F^e_\ast R} {F^e}^! R \cong M \otimes_R {F^e}^! R$ by \cite[Lemma 2.5]{blickleboecklecartierfiniteness}. Since ${F^e}^! R = \Hom_R(F^e_\ast R, R)$ the claimed equivalence is just a case of Morita equivalence (cf.\ e.\,g.\ \cite[Theorem 18.24]{lamlecturesonmodules}).

As ${F^e}^!$ induces an equivalence it is fully faithful and hence reflects isomorphisms.
\end{proof}

\section{Bernstein-Sato polynomials}\label{sec.BernSato}

In this section we introduce our notion of Bernstein-Sato polynomial after transfering some results of Musta\c{t}\u{a} in \cite{mustatabernsteinsatopolynomialspositivechar} to our right $\mathcal{D}$-module situation.

If $k$ is perfect and $\Spec R$ is smooth most of the results in this section follow formally once one observes that given a right $\mathcal{D}_R[t]$-module the operator $t^{p^e} \partial_t^{[p^e]}$ acts via $-\partial_t^{[p^e]} t^{p^e}$ on the left module obtained by tensoring with $\omega^{-1}_{R[t]/k}$ (\cite[1.3.4]{berthelotdmodulesarithmetiqueii}).

\begin{Le}
\label{EigenspaceDecompExistence}
Let $R$ be regular and $F$-finite and let $R[t]$ be the polynomial ring in one variable over $R$. Given a right $\mathcal{D}^e_{R[t]}$-module $M$, there is a unique decomposition as $\mathcal{D}^e_R$-modules
\[M = \bigoplus_{i \in \mathbb{F}_p^e} M_{i},\] where for $1 \leq l \leq e$ the operator $\theta_{p^{l-1}}$ acts on $M_{i}$ via $i_l$. This decomposition is preserved by $\mathcal{D}^e_{R[t]}$-morphisms. The same statement holds for $\mathcal{D}^e_R[\theta_1, \ldots, \theta_{p^{e-1}}]$-modules.
\end{Le}
\begin{proof}
This works similarly to \cite[Proposition 4.2]{mustatabernsteinsatopolynomialspositivechar}. More precisely, one has \[\prod_{j=0}^{p-1} (\theta_{p^e} + j) = \prod_{j=1}^{p} (\theta_{p^e} + j) = 0\] for all $e \geq 0$. Indeed, by Lemma \ref{DiffOpFormulae} it suffices to show that $(\partial_t^{[p^e]})^p  = 0$. This in turn follows from (b) since $\frac{p^{e+1}!}{(p^e!)^p} $ is divisible by $p$. Using this and the fact that $[\theta_i, \theta_j] = 0$ the existence of such a decomposition follows. The remaining statements follow easily.
\end{proof}

Given $M$ we refer to the decomposition of Lemma \ref{EigenspaceDecompExistence} as the \emph{eigenspace decomposition of $M$ (with respect to the Euler operators)}.

Note that if $M$ is a right $\mathcal{D}^e_{R[t]}$-module then it is in particular a right $\mathcal{D}^{e-1}_{R[t]}$-module. Hence, $M$ admits eigenspace decompositions with respect to $\mathcal{D}^e_{R[t]}$ and with respect to $\mathcal{D}^{e-1}_{R[t]}$ and these are compatible. That is, if $M_{(i_1, \ldots, i_{e-1})}$ is an eigenspace for the $\theta_{p^l-1}$ with $1 \leq l \leq e -1$ then
\[ M_{(i_1, \ldots, i_{e-1})} = \bigoplus_{j \in \mathbb{F}_p} M_{(i_1, \ldots, i_{e-1}, j)} \] is an eigenspace decomposition with respect to the $\theta_{p^l-1}$ for $1 \leq l \leq e$. Again a similar statement holds for right $\mathcal{D}^e_R[\theta_1, \ldots, \theta_{p^{e-1}}]$-modules.

Given a morphism $f: \Spec S \to \Spec R$ of regular schemes and a right $\mathcal{D}_S$ module $M$ one defines the pushforward $f_+ M$ as $f_\ast (M \otimes_{\mathcal{D}_S} S) \otimes_R \mathcal{D}_R$. We thus have a natural map $f_\ast M \to f_+ M$. By abuse of notation we will denote the image of $f_\ast M$ under this map again by $f_\ast M$. Similarly, we define the pushforward $f_+ M$ for a right $\mathcal{D}^e_S$-module as $f_\ast(M \otimes_{\mathcal{D}^e_S} S) \otimes_R \mathcal{D}^e_R$.

Our next goal is to describe this natural map in the setting where we identify $\mathcal{D}^e_R$ with $\End_R(F_\ast^e R)$. We first need a

\begin{Le}
\label{PushforwardIsoLemma}
Let $R$ be regular and $F$-finite and $M$ a (right) $R$-module. Then $F_\ast^e M \otimes_{F_\ast^e R} \End_R(F_\ast^e R) \to \Hom_R(F_\ast^e R, M), m \otimes \varphi \mapsto [ r \mapsto m \varphi(r)]$ is an isomorphism of $\End_R(F_\ast^e R)$-modules.
\end{Le}
\begin{proof}
The module structure on $F_\ast^e M \otimes_{F_\ast^e R} \End_R(F_\ast^e R)$ is given by multiplication from the right and the one on $\Hom_R(F_\ast^e R,M)$ is given by premultiplication.
Bijectivity is local so that we may assume that $F_\ast^e R$ is a free $R$-module (use \cite[Theorem 2.1]{kunzpositive}). Fix a basis $b_1, \ldots, b_n$ and let $\delta_i: F_\ast^e R \to F_\ast^e R$ be the projection onto the $i$th basis vector. We define a map $\Hom_R(F_\ast^e R, M) \to F_\ast^e M \otimes_{F_\ast^e R} \End_R(F_\ast^e R), \varphi \mapsto \sum_{i} \varphi(b_i) \otimes \delta_i$ which is a two-sided inverse.
\end{proof}

\begin{Prop}
\label{DModPushforward}
Let $f: \Spec S \to \Spec R$ be a morphism of regular $F$-finite schemes and let  $M$ be an $S$-module. Let $\Hom_S(F_\ast^e S, M)$ be a right $\mathcal{D}_S^e$-module via the action on the first factor. Then $f_+ \Hom_S(F_\ast^e S, M)$ is naturally isomorphic to ${F^e}^! f_\ast M$. Under this identification the natural map \[f_\ast \Hom_S(F_\ast^e S, M) \to f_+ \Hom_S(F_\ast^e S, M)\] is given by the composition of the canonical maps \[ f_\ast \Hom_S(F_\ast^e S, M) \to \Hom_{R}(f_\ast F_\ast^e S, f_\ast M) \to \Hom_{R}(F_\ast^e f_\ast S, f_\ast M)\] with the map \[\Hom_{R}(F_\ast^e f_\ast S, f_\ast M) \to \Hom_R(F^e_\ast R, f_\ast M), \quad\varphi \mapsto \varphi \circ F^e_\ast(f^\#).\]
\end{Prop}
\begin{proof}
By definition the pushforward $f_+ \Hom_S(F_\ast^e S, M)$ is given as \[F_\ast^e f_\ast (\Hom_S(F_\ast^e S, M) \otimes_{\End_S(F_\ast^e S)} F_\ast^e S) \otimes_{F_\ast^e R} \End_{R}(F_\ast^e R).\] Note that the right $F_\ast^e R$-module structure on the term $F_\ast^e f_\ast(\ldots)$ is given by $(\psi \otimes s) \cdot r = \psi \otimes s f(r)$. By Proposition \ref{EquivalenceModRModD} the whole expression is isomorphic to $F_\ast^e f_\ast M \otimes_{F_\ast^e R} \End_R(F_\ast^e R)$. Using Lemma \ref{PushforwardIsoLemma} above we obtain that $F_\ast^e f_\ast M \otimes_{F_\ast^e R} \End_R(F_\ast^e R) \to \Hom_R(F_\ast^e R, f_\ast M)$ is an isomorphism. One readily checks that the natural map is given by the formula above.
\end{proof}

Next we show that this isomorphism is compatible with direct limits. We first need a general Lemma.

\begin{Le}
\label{LimitLemma}
Let $I$ be a directed system and let $M_i, N_i$ be $R_i$-modules such that $M_i, N_i$ and $R_i$ are filtered by $I$. Write $M = \colim_i M_i$, $N = \colim_i N_i$ and $R = \colim_i R_i$. Then one has an isomorphism $\colim_{i} (M_i \otimes_{R_i} N_i) \to M \otimes_R N$.
\end{Le}
\begin{proof}
We have $R_i$-bilinears map $M_i \times N_i \to M \otimes_R N$ which induce, by the universal properties of the limit and the tensor product, an $R$-linear map $\colim_i M_i \otimes_{R_i} N_i \to M \otimes N, [m_i \otimes n_i] \mapsto [m_i] \otimes [n_i]$.

On the other hand, the maps $M_i \times N_i \to \colim_i M_i \otimes_{R_i} N_i, (m_i, n_i) \mapsto [m_i \otimes n_i]$ induce an $R$-bilinear map $M \times N \to \colim_i M_i \otimes_{R_i} N_i$. This in turn induces an $R$-linear map $M \otimes N \to \colim_i M_i \otimes_{R_i} N_i$ which is an inverse to the map constructed above.
\end{proof}

\begin{Prop}
\label{DmodpushforwardlimitCartier}
Let $f: \Spec S \to \Spec R$ be a morphism of regular $F$-finite schemes and assume that $(M, \kappa)$ is a Cartier module on $\Spec S$ and denote its limit over the $C^e$ by $\mathcal{M}$. Then $f_+ \mathcal{M}$ is naturally isomorphic to $\colim {F^e}^! f_\ast M \cong \colim f_+ {F^e}^! M$.
\end{Prop}
\begin{proof}
The first claimed isomorphism obtained by applying Lemma \ref{LimitLemma} twice.

For the second claimed isomorphism note that by Proposition \ref{DModPushforward} we have isomorphisms $f_+ {F^e}^! M \to {F^e}^! f_\ast M$. Moreover, we have a commutative diagram
\[ \begin{xy} \xymatrix{f_+ {F^e}^! M \ar[r] \ar[d] & \Hom_R(F_\ast^e R, f_\ast M) \ar[d]^{\varphi \mapsto [ r \mapsto \kappa(\varphi(r))]}\\  f_+ {F^{e+1}}^! M \ar[r]& \Hom_R(F_\ast^{e+1} R, f_\ast M),}\end{xy}\] where the left vertical map is given by tensoring $\varphi \mapsto [s \mapsto \kappa(\varphi(s))]$ with the natural maps $F^e_\ast S \to F^{e+1}_\ast S$ and $\mathcal{D}^e_S \to \mathcal{D}^{e+1}_S$. 
\end{proof}

\begin{Le}
\label{BSPQuotientDModule}
Let $R$ be regular and $F$-finite and $f \in R$.
Denote by $\gamma: \Spec R \to \Spec R[t]$ the graph embedding along $f$ and let $M$ be a right $\mathcal{D}_R$-module. Then the quotient \[N := (\gamma_\ast M) \mathcal{D}^e_R[\theta_1, \theta_p, \ldots, \theta_{p^{e -1}}] / (\gamma_\ast M) \mathcal{D}^e_R[\theta_1, \theta_p, \ldots, \theta_{p^{e -1}}] t\] is a right $\mathcal{D}^e_R[\theta_1, \theta_p, \ldots, \theta_{p^{e-1}}]$-module.
\end{Le}
\begin{proof}
The claim follows from Lemma \ref{DiffOpFormulae} (i) and $(\gamma_\ast M) t = \gamma_\ast(Mf)$.
\end{proof}

\begin{Def}
\label{BSPDefinition}
With the notation of Lemma \ref{BSPQuotientDModule} let $\Gamma^e_f \subseteq \{0, \ldots, p-1\}^e$ be the set of those $i = (i_1, \ldots, i_e) \in \mathbb{F}_p^e$ for which the eigenspace $N_i$ of $N$ (as constructed in Lemma \ref{EigenspaceDecompExistence}) is non-trivial. Then we define the $e$th \emph{Bernstein-Sato polynomial} of $M$ as \[b^e_{M,f}(s) = \prod_{i \in \Gamma^e_f} (s - (\frac{i_e}{p} + \cdots + \frac{i_1}{p^e})) \in \mathbb{Q}[s], \] where we lift\footnote{Here and elsewhere we will always view $\mathbb{F}_p$ as $\{0, \ldots, p-1\}$ and confuse elements in $\mathbb{F}_p$ with a lift whenever this is convenient.} elements of $\mathbb{F}_p = \{0, \ldots, p-1\}$ to $\mathbb{Z}$.
\end{Def}

Note that since we are working with right modules rather then left modules we do not need to invert the sign of the eigenvalues as in \cite{mustatabernsteinsatopolynomialspositivechar}. More precisely, Musta\c{t}\u{a} considers the action of $\partial^{[p^l]} t^{p^l}$ on the left $\mathcal{D}_R^e[\vartheta_1, \ldots, \vartheta_{p^{e-1}}]$-module \[\mathcal{D}_R^e[\vartheta_1, \ldots, \vartheta_{p^{e-1}}]\gamma_\ast R/ t \mathcal{D}_R^e[\vartheta_1, \ldots, \vartheta_{p^{e-1}}]\gamma_\ast R\] and if $(i_1, \ldots, i_e)$ is an eigenvalue for the left action then $\frac{-i_e}{p} + \ldots + \frac{-i_1}{p^e}$ is encoded as a zero in a Bernstein-Sato polynomial. As pointed out at the beginning of Section \ref{sec.BernSato} these Bernstein-Sato polynomials coincide with the one defined in \ref{BSPDefinition} provided $R$ is smooth over a perfect field $k$. In fact, Theorem \ref{MainResult} and \cite[Theorem 6.7]{mustatabernsteinsatopolynomialspositivechar} show that the polynomials coincide for $R$ regular essentially of finite type over an $F$-finite field and $e \gg 0$.

\begin{Bem}
We comment on our definition of Bernstein-Sato polynomial and its relation to the definition over the complex numbers. Let $X = \mathbb{A}^{n+1}_\mathbb{C}$. Then $\mathcal{D}_X$ is just the Weyl algebra $\mathbb{C}[x_1, \ldots, x_{n+1}, \partial_1, \ldots, \partial_{n+1}]$ with the usual relation $[\partial_i, x_j] = \delta_{ij}$. Assume that the hypersurface is given by $t = x_{n+1}$. Then for a regular holonomic quasi-unipotent $\mathcal{D}_X$-module $M$ the $V$-filtration along $t$ is a decreasing $\mathbb{Q}$-indexed filtration with certain properties (see \cite{budurvfiltrationdmodules} for a definition).

In particular, one has (cf. \cite[Proposition 2.1.7]{sabbahvfiltrations}) $V^k M = V^k(\mathcal{D}_R) V^0(M)$ for $k \leq 0$ but $V^k M = V^{k-1}(\mathcal{D}_R) V^1(M)$ only for $k \geq 1$.\footnote{Note that following \cite{budurvfiltrationdmodules} we have inverted signs here compared to \cite{sabbahvfiltrations}.}
Here $V^k(\mathcal{D}_X)$ is the $V$-filtration on $\mathcal{D}_X$ which is given by \[V^k(\mathcal{D}_X) = t^k V^0(\mathcal{D}_X) \quad \text{ for }k \geq 1\] and by \[V^k(\mathcal{D}_X) = V^0(\mathcal{D}_X) \mathcal{D}_{X, -k} \quad \text{ for } k \leq -1,\] where $\mathcal{D}_{X, -k}$ are the differential operators of order $\leq -k$. Finally, $V^0(\mathcal{D}_X)$ is given by $\{ \sum_{n \geq m} f(x) \partial_x^\alpha t^n \partial_t^m \, \vert \, f \in \mathbb{C}[x_1, \ldots, x_n]\}$.

The Bernstein-Sato polynomial of $M$ is now defined as the monic minimal polynomial $b \in \mathbb{C}[s]$ such that $b(t \partial_t +k ) V^k M \subseteq V^{k+1} M$ for all $k \in \mathbb{Z}$. By the above it is actually sufficient to construct a Bernstein-Sato polynomial that satisfies the above identity for $k = -1, 0, 1$.

Since in characteristic $p > 0$ the Brian\c{c}on-Skoda theorem \cite[Theorem 4.21]{blicklep-etestideale} always yields $f^n \tau(M, f^{t - n}) = \tau(M, f^t)$ for $t \geq n$ our definition can be seen as an analogue of the one over the complex numbers since we only need to control the range $k = 0$.
\end{Bem}

\section{Test modules and $\mathcal{D}$-modules}
\label{SectionTestDmodules}
In this section we relate test modules of a Cartier module $M$ over a regular $F$-finite ring $R$ with certain right $\mathcal{D}^e_R$ submodules of ${F^e}^! M$. First, we need several technical lemmata concerning test modules.

\begin{Le}
\label{CartierIterationAscendingChain}
Let $R$ be essentially of finite type over an $F$-finite field, $(M, \kappa)$ an $F$-pure coherent Cartier module, $t \in \mathbb{Q}_{\geq 0}$ and let $f$ be an $M$-regular element. Then one has $\kappa^e f^{\lceil tp^e\rceil} f^l M \subseteq \kappa^{e+1} f^{\lceil t p^{e+1} \rceil} f^l M$ for all $l \geq 0$. In particular, for $e \gg 0$ depending on $l$ equality holds.
\end{Le}
\begin{proof}
First of all, for any $l$ we have \[\kappa^n f^lM \supseteq \kappa^n f^{lp} M = \kappa^{n-1} f^{l} \kappa M = \kappa^{n-1} f^{l} M,\] where we used that $\kappa$ is surjective since $M$ is $F$-pure.

Next, we have \[\kappa f^{\lceil tp^{e+1} \rceil} f^lM \supseteq \kappa (f^{\lceil tp^e\rceil})^p f^l M = f^{\lceil tp^e\rceil} \kappa f^lM \supseteq f^{\lceil tp^e\rceil}f^l M,\] where we used the previous observation with $n =1$ for the last inclusion. Applying $\kappa^e$ on both sides yields the claimed inclusion. Since $M$ is coherent the ascending chain $\kappa^e f^{\lceil tp^e\rceil} f^l M$ stabilizes.
\end{proof}

\begin{Bem}
With the notation of Lemma \ref{CartierIterationAscendingChain} if $t = \frac{m}{p^s}$ and $n \geq s$ is such that $\kappa^n f^l M = M$ then equality holds for any $e \geq 2n$. Indeed,
\[ \kappa^e f^{mp^{e -s}} f^l M = \kappa^e (f^{mp^{n-s}})^{p^{e-n}} f^l M = \kappa^{n} f^{mp^{n-s}} \kappa^{e-n} f^l M = \kappa^n f^{mp^{n-s}} M\] which is independent of $e$.

If $R$ is a polynomial ring and $M$ is given explicitly by a presentation $R^a \to R^b \to M$ then we expect that it should be possible to determine $e$ explicitly for a given hypersurface $f$.
\end{Bem}

\begin{Le}
\label{FPureHypersurface}
Let $R$ be essentially of finite type over an $F$-finite field, let $(M, \kappa)$ be an $F$-regular coherent Cartier module, $t \in \mathbb{Z}[\frac{1}{p}]$ and let $f$ be an $M$-regular element. Consider the Cartier algebra $\mathcal{C}$ generated in degree $e \geq 1$ by $\kappa^e f^{\lceil t p^e\rceil}$. Then $\underline{M}_\mathcal{C} = \kappa^e f^{\lceil tp^e \rceil} M$ for all $e \gg 0$.
\end{Le}
\begin{proof}
The calculation 
\[
\kappa^{a} f^{\lceil tp^a \rceil} \kappa^{a'} f^{\lceil tp^{a'} \rceil} = \kappa^{a+a'} f^{{\lceil tp^a \rceil}p^{a'}} f^{\lceil tp^{a'} \rceil}  = \kappa^{a + a'} f^{\lceil t p^{a+a'} \rceil} f^{r} 
\]
with $0 \leq r = \lceil tp^{a'} \rceil +  \lceil tp^a\rceil p^{a'} - \lceil t p^{a + a'}\rceil$ shows that $\mathcal{C}$ is in indeed a Cartier algebra. Applying Lemma \ref{CartierIterationAscendingChain} with $l=0$ shows that $\kappa^ef^{\lceil tp^e \rceil}M \subseteq \kappa^{e+1}f^{\lceil tp^{e+1} \rceil}M$ (with equality for $e \gg 0$) so that for all $e \gg 0$ we have $\kappa^ef^{\lceil tp^e \rceil}M = \mathcal{C}_+M$.

By \cite[Proposition 2.13 and Corollary 2.14]{blicklep-etestideale} one has $\underline{M}_{\mathcal{C}} = (\mathcal{C}_+)^h M$ for all $h \gg 0$. Fix such an $h$ and $e$ as above. Then the inclusion from left to right follows by the same argument as above.

For the other inclusion we have to prove that $\mathcal{C}_+ \kappa^e f^{\lceil tp^e \rceil} M \supseteq \kappa^e f^{\lceil tp^e \rceil} M$ for $e \gg 0$. We use the assumption on $t$ and write $t = \frac{m}{p^s}$ with $m \in \mathbb{Z}$ and assume that $e \geq s$. We consider elements of the form $\kappa^{e'} f^{\lceil tp^{e'}\rceil} \in \mathcal{C}_+$. For $e' \geq e$ we compute \[ \kappa^{e'} f^{mp^{e'-s}} \kappa^e f^{mp^{e-s}} M = \kappa^{e} f^{mp^{e-s}} \kappa^{e'-e} \kappa^e f^{mp^{e-s} } M = \kappa^{e} f^{mp^{e-s}}M,\] where for the last equality we used the $F$-regularity of $(M, \kappa)$ (see \cite[Proposition 5.2]{staeblertestmodulnvilftrierung} and Lemma \ref{CartierIterationAscendingChain}) possibly choosing a larger $e'$.
\end{proof}

Next we prove a variant of \cite[Lemma 2.1]{blicklemustatasmithdiscretenesshypersurfaces} for modules.

\begin{Le}
\label{TestModuleComputationHypersurface}
Let $R$ be essentially of finite type over an $F$-finite field. Let $(M, \kappa)$ be an $F$-regular coherent Cartier module, $f \in R$ a non zero-divisor on $M$ and $t \in \mathbb{Z}[\frac{1}{p}]$. Then for all $e \gg 0$ we have $\tau(M, f^t) = \kappa^e (f^{t p^{e}}M)$.
\end{Le}
\begin{proof}
First of all, note that $M_f$ is $F$-regular with respect to the Cartier algebra $\mathcal{C}$ generated by the $\kappa^e f^{\lceil t p^e\rceil}$. By \cite[Theorem 3.11]{blicklep-etestideale} we thus have \[\tau(M, f^t) = \sum_{n \geq 1} \mathcal{C}_n f \underline{M}_\mathcal{C}.\] By Lemma \ref{FPureHypersurface} we may replace $\underline{M}_\mathcal{C}$ by  $\kappa^e f^{tp^e} M$ for any sufficiently large $e$. As seen in the proof of Lemma \ref{FPureHypersurface} we have $\mathcal{C}_n=\kappa^n f^{\lceil tp^n\rceil}R$.

We compute \[\kappa^n f^{\lceil tp^n \rceil} f \kappa^e f^{tp^e} M = \kappa^{n+e} f^{\lceil t p^{n}\rceil p^e} f^{(t+1)p^e} M \subseteq \kappa^{n+e} f^{\lceil t p^{n+e} \rceil} f^{(t+1)p^e} M\] and for $t p^n \in \mathbb{Z}$ equality holds. In particular, using Lemma \ref{CartierIterationAscendingChain} we obtain that \[\sum_{n \geq 1} \kappa^n f^{\lceil tp^n \rceil} f \kappa^e f^{tp^e}  M = \kappa^n f^{ tp^n} f\kappa^e f^{tp^e}  M \text{ for some } n \gg 0.\]
Finally, we have for $n \gg 0$ and suitable $e'$
\[\kappa^n f^{ tp^n} f\kappa^e f^{tp^e}  M = \kappa^{n+e - e'} f^{t p^{n+e - e'}} \kappa^{e'} f^{(t+1)p^e} M = \kappa^{n+e -e'} f^{t p^{n+e - e'}} M,\] where the last equality is due to the $F$-regularity of $(M, \kappa)$.
\end{proof}

\begin{Bem}
\label{EffectiveComputation}
In particular, if $t = \frac{m}{p^s} \in \mathbb{Z}[\frac{1}{p}]$ we may write $t = \frac{mp^{e-s}}{p^e}$ so that $\tau(M, f^t) = \kappa^e ( f^{m'} M)$, where $m' = mp^{e-s}$.

Also note that by right-continuity (i.\,e.\ $\tau(M, f^t) = \tau(M, f^{t + \eps})$ for small $\eps > 0$) we may always assume that $t \in \mathbb{Z}[\frac{1}{p}]$ if we want to compute test modules.
\end{Bem}

Given a $\kappa$-module $M$ we have a natural map (the adjoint of $\kappa$) $C: M \to F^!M, m \mapsto (r \mapsto \kappa(rm))$ and if $N$ is an $R$-submodule of $M$ we may consider the right $\mathcal{D}^e_R$-submodule of ${F^e}^! M$ generated by $N$ which by definition is $C^e(N) \cdot \mathcal{D}^e_R$.

\begin{Le}
\label{InducedDModule}
Let $R$ be regular and $F$-finite. Let $M$ be a coherent $\kappa$-module and let $N$ be an $R$-submodule of $M$. Then ${F^e}^! \kappa^{e}(F_\ast^e N)$ is the right $\mathcal{D}_R^e$-submodule of ${F^e}^! M$ generated by the image of $N$ in ${F^e}^! M $.
\end{Le}
\begin{proof}
Clearly, ${F^e}^!\kappa^e (F_\ast^e N)$ is a right $\mathcal{D}^e_R$-submodule of ${F^e}^! M$. So one inclusion is dealt with once we show that $C^e(N) \subseteq {F^e}^! \kappa^e(F_\ast^e N)$. If $n \in N$ then $C^e(n): F^e_\ast R \to M, r \mapsto \kappa^e(rn)$ and $C^e(rn) \in \kappa^e(F_\ast^e N)$ for all $r \in F_\ast^e R$.

For the other inclusion we may assume that $R$ is local. Hence, $F_\ast^e R$ is free of rank $s = p^{e \dim R}$ since $R$ is regular and $F$-finite (see \cite[Theorem 2.1]{kunzpositive}). Fix a basis $b_1, \ldots, b_s$ of $F_\ast^e R$ and let $\varphi: F_\ast^e R \to \kappa^e(F_\ast^e N)$ be an element of ${F^e}^! \kappa^e(F_\ast^e N)$. Each $\varphi(b_i)$ is of the form $\kappa^e(n_i)$ for some $n_i \in N$. We then can write $\varphi$ as $\sum_{i=1}^s \psi_{i} \circ p_i$, where $p_i: F_\ast^e R \to F_\ast^e R, p_i(b_j) = \delta_{ij}$ and $\psi_i(r) = \kappa^e(r n_i)$. Then $p_i \in \mathcal{D}^e_R$ and $\psi_i \in C^e(N)$ as desired.
\end{proof}

\begin{Bsp} Let $M$ be a Cartier module on $R$ and let $\gamma: \Spec R \to \Spec R[t]$ be the graph embedding along $f$. Then one has, in general, a proper inclusion $\gamma_\ast(C^e(M)) \mathcal{D}^e_R \subseteq \gamma_\ast({C^e}(M) \mathcal{D}^e_R) \mathcal{D}^e_R$. Here we let $\mathcal{D}^e_R$ act via the inclusion $\mathcal{D}^e_R \to \mathcal{D}^e_{R[t]}$.
Also note that $C^e(M) \mathcal{D}_R^e = {F^e}^! M$.

As an example for this let $R = k[y]$ and $M = R$ with twisted Cartier structure $\kappa \cdot y$, where $\kappa(y^i) = \delta_{i (p-1)}$ for $0 \leq i \leq p-1$. Then ${F^e}^! R$ contains $\kappa^e$ and we claim that this is not contained in $\gamma_\ast(C^e(M)) \mathcal{D}^e_R$. Indeed, any element $\psi$ in $\gamma_\ast(C^e(M)) \mathcal{D}^e_R$ is of the form
\[\psi(y^i t^j) = \sum_{l =0}^s \kappa(y^{\frac{p^e -1}{p-1}} r_l P_l(y^i) y^j),\]
where $P_l \in \mathcal{D}^e_R$ and we used that $(\kappa y)^e = \kappa^e y^{1 + p + \ldots + p^{e-1}}$.
Here $\kappa^e$ is given by acting on basis elements $y^i t^j \mapsto 1$ if $i+j = p^e - 1$ and $0$ else. To see the claim let now $i = 0$ and $j = p^e-1$. Then we can write $\psi(t^{p^e -1})$ as $y \cdot \sum \ldots$ and clearly this cannot evaluate to $1$. Hence, the inclusion is strict.
\end{Bsp}

Combining the two previous lemmata we obtain the following

\begin{Ko}
\label{TestModuleDModuleEquivalence}
Let $R$ be regular and essentially of finite type over an $F$-finite field. Let $(M, \kappa)$ be an $F$-regular Cartier module and let $f \in R$ be a non zero-divisor on $M$ and $t \in \mathbb{Z}[\frac{1}{p}]$. Then for $e \gg 0$ and $t = \frac{m}{p^e}$ we have ${F^e}^! \tau(M, f^{t}) = C^e(f^{m} M) \cdot \mathcal{D}^e_R$.
\end{Ko}
\begin{proof}
Immediate from Lemma \ref{TestModuleComputationHypersurface} and Lemma \ref{InducedDModule}.
\end{proof}

\section{Test modules and Bernstein-Sato polynomials}

In this section we prove the main result of this paper. That is, we show that the roots of the Bernstein-Sato polynomials for $e$ sufficiently large are precisely the $F$-jumping numbers in the range $(0,1]$ of the test module filtration.

Recall that given a morphism $\gamma: \Spec R \to \Spec R[t]$ and a right $\mathcal{D}^e_R$-module $M$ we denote the image of the natural map $\gamma_\ast M \to \gamma_+ M$ again by $\gamma_\ast M$. In particular, $\gamma_\ast M$ is then contained in $\Hom_{R[t]}(F_\ast^e R[t], \gamma_\ast M)$ (see Proposition \ref{DModPushforward}).

\begin{Le}
Let $R$ be a regular $F$-finite ring and $(M, \kappa)$ a Cartier module with adjoint $C$. Let $f \in R$ and denote by $\gamma: \Spec R \to \Spec R[t]$ the graph embedding along $f$. Then we have an  eigenspace decomposition \[(\gamma_\ast C^e(M)) \mathcal{D}_R^e[\theta_1, \ldots, \theta_{p^{e-1}}] = \bigoplus_{i \in \mathbb{F}_p^e} (\gamma_\ast C^e(M)) \mathcal{D}_R^e \circ \pi_{i},\] where \[\pi_i: F^e_\ast R[t] \to F^e_\ast R[t],\quad \pi_i(rt^j) =\begin{cases} rt^j,\, j = \sum_{l = 1}^{e} i_l p^{l-1},\\
                0,\, \text{else}
               \end{cases}\]  is the projection onto the eigenspace.
\end{Le}
\begin{proof}
We denote $(\gamma_\ast C^e(M)) \mathcal{D}_R^e[\theta_1, \ldots, \theta_{p^{e-1}}]$ by $N$ and its eigenspaces by $N_i$. First of all, note that Lucas' theorem shows that $\pi_i$ is the projection onto the eigenspace $i_1, \ldots, i_e$ for the $\theta_1, \ldots, \theta_{p^{e-1}}$.

Assume that $\varphi \in N_i$. Then viewing $\varphi$ as an element of $\Hom_{R[t]}(F_\ast^e R[t], \gamma_\ast M)$ we have $\varphi(rt^m) = 0$ for $m \neq i_1 + i_2p + \ldots + i_e p^{e-1}$ so that $\varphi$ factors through $\pi_i$.

In the other direction note that given $\varphi \in N$ we have $\varphi = \sum_{i\in \mathbb{F}_p^e}\varphi \circ \pi_i$ which is a decomposition of $\varphi$ in the ambient module $(\gamma_\ast F^! M)\mathcal{D}^e_R[\theta_1, \ldots, \theta_{p^{e-1}}]$ as eigenvectors. Since this decomposition is preserved by morphisms (see Lemma \ref{EigenspaceDecompExistence}) the above has to be the decomposition of $\varphi$ in $N$ as well.
\end{proof}

\begin{Prop}
\label{EigenspaceIsomorphism}
Let $R$ be a regular $F$-finite ring and $(M, \kappa)$ a Cartier module with adjoint $C$. Let $f \in R$ and denote by $\gamma: \Spec R \to \Spec R[t]$ the graph embedding along $f$. Then the $(i_1, \ldots, i_e)$-eigenspace of $(\gamma_\ast C^e(M)) \mathcal{D}_R^e[\theta_1, \ldots, \theta_{p^{e-1}}]$ is isomorphic to $C^e(f^{i_1 + i_2 p + \ldots i_ep^{e-1}} M) \cdot \mathcal{D}^e_R$ as a right $\mathcal{D}^e_R$-module.
\end{Prop}
\begin{proof}
We write $m = i_1 + i_2 p + \ldots + i_e p^{e-1}$. Clearly, the right $\mathcal{D}_R^e$-submodule $F_\ast^e R \cdot t^{m}$ of $F_\ast^e R[t]$ is isomorphic to $F_\ast^e R$ by sending $t^{m}$ to $1$. In particular, we have a commutative diagram
\[
 \begin{xy} \xymatrix{F_\ast^e R t^m  \ar[d] \ar[r]^{\varphi} & \gamma_\ast M\\
F_\ast^e R \ar[ru]_{\varphi( f^m )}}
 \end{xy}
\]
for any element $\varphi$ of the eigenspace. This induces the desired isomorphism of right $\mathcal{D}^e_R$-modules.
\end{proof}

\begin{Ko}
\label{QuotientEigenspaceIsomorphism}
The quotient \[(\gamma_\ast C^e(M))\mathcal{D}^e_R[\theta_1, \ldots, \theta_{p^e-1}]/(\gamma_\ast C^e(fM)) \mathcal{D}^e_R[\theta_1, \ldots, \theta_{p^e-1}]\] is a right $\mathcal{D}^e_R[\theta_1, \ldots, \theta_{p^e-1}]$-module and the $(i_1,\ldots, i_e)$-eigenspace of the quotient is isomorphic to \[C^e(f^{i_1 + i_2 p + \ldots + i_ep^{e-1}} M) \cdot \mathcal{D}^e_R/C^e(f^{1+i_1 + i_2 p + \ldots + i_ep^{e-1}} M) \cdot \mathcal{D}^e_R\] as a right $\mathcal{D}^e_R$-module.
\end{Ko}
\begin{proof}
The first claim is just Lemma \ref{BSPQuotientDModule}. Since the eigenspace decomposition is preserved by the canonical projection we get the desired isomorphism on the quotient by Proposition \ref{EigenspaceIsomorphism}.
\end{proof}

Following \cite{mustatabernsteinsatopolynomialspositivechar} we introduce some notation. Namely, given $\lambda \in (0,1]$ we can write it uniquely as \[\lambda = \sum_{i \geq 1} \frac{c_i(\lambda)}{p^i}\] with all $c_i(\lambda) \in \{0, \ldots, p-1\}$, and such that infinitely many of them are non-zero. Moreover, one obtains for every $e \geq 1$ that \[\sum_{i=1}^e\frac{c_i(\lambda)}{p^i} = \frac{\lceil \lambda p^e\rceil -1}{p^e}.\]

We are now ready to state and prove our main result:

\begin{Theo}
\label{MainResult}
Let $R$ be regular essentially of finite type over an $F$-finite field. Let $(M, \kappa)$ be an $F$-regular Cartier module and $f \in R$ a non zero-divisor on $M$. The roots of the Bernstein-Sato polynomials $b^e_{M,f}(s)$ are given for $e$ sufficiently large by $\frac{\lceil \lambda p^e \rceil -1}{p^e}$, where $\lambda$ varies over the $F$-jumping numbers of the test module filtration $\tau(M, f^t)$ for $t \in (0,1]$.
\end{Theo}
\begin{proof}
By definition $\lambda \in (0,1]$ is an $F$-jumping number if and only if for $e \gg 0$ we have \[\tau(M, \kappa, f^{\frac{\lceil \lambda p^e \rceil -1}{p^e}}) \neq \tau(M, \kappa, f^{\frac{\lceil \lambda p^e \rceil}{p^e}}).\] Using the fact that ${F^e}^!$ is fully faithful (Proposition \ref{EquivalenceModRModD}) this inequality is equivalent to \[{F^e}^! \tau(M ,\kappa, f^{\frac{\lceil \lambda p^e \rceil -1}{p^e}}) \neq {F^e}^! \tau(M ,\kappa, f^{\frac{\lceil \lambda p^e \rceil}{p^e}})\] for all $e \gg 0$. We write $\lceil \lambda p^e\rceil -1 = i_1 + i_2 p + \ldots +i_e p^{e-1}$. Then by Corollary \ref{TestModuleDModuleEquivalence} the above means that \[C^e(f^{i_1 + i_2 p + \ldots + i_ep^{e-1}} M) \cdot \mathcal{D}^e_R \neq C^e(f^{1+i_1 + i_2 p + \ldots + i_ep^{e-1}} M) \cdot \mathcal{D}^e_R\] for all $e \gg 0$. Finally, by Corollary \ref{QuotientEigenspaceIsomorphism} this is equivalent to $\frac{\lceil \lambda p^e \rceil -1}{p^e}$ being a zero of $b^e_{M,f}(s)$ for all $e \gg 0$.
\end{proof}

\begin{Bem}
\begin{enumerate}[(a)]
\item{We recall that there are only finitely many $F$-jumping numbers in $(0,1]$ and that they are all rational. In particular, the limit over the $b^e_{M,f}(s)$ for $e \to \infty$ is a polynomial with rational roots.}
 \item{The case where $M$ is locally constant (and $R$ is smooth over a perfect field), i.\,e.\ when there exists a finite \'etale morphism $\varphi: \Spec S \to \Spec R$ such that $\varphi^\ast M \cong \omega_S^n$ can also be directly deduced from the constant case treated in \cite{mustatabernsteinsatopolynomialspositivechar}. This essentially boils down to the fact that differential operators along \'etale morphisms are well-behaved (one obtains an inclusion $\mathcal{D}^e_R \to \mathcal{D}^e_S$ and the natural map $M \to \varphi^! M$ is $\mathcal{D}^e_R$-linear -- see \cite[Theorem 2.2.5, Corollary 2.2.6]{massonthesis} for the first statement. The latter may be extracted from \cite[Theorem 2.2.10]{massonthesis} and \cite[2.1.3]{berthelotdmodulesarithmetiqueii}). Then one uses \cite[Theorem 8.5]{staeblertestmodulnvilftrierung} to see that the $F$-jumping numbers of $M$ are the same as that of $\omega_S$. In fact, writing this up precisely was the original motivation for this paper.}
\item{Note that Musta\c{t}\u{a}'s result (\cite[Theorem 6.7]{mustatabernsteinsatopolynomialspositivechar}) is valid for any $e \geq 1$ while we only obtain a result for $e \gg 0$.}
\end{enumerate}

\end{Bem}

\section{A comparison with Stadnik's $b$-functions}
\label{stadniksbfunctions}

The goal of this section is to point out the relation of our Bernstein-Sato polynomials to the $b$-functions of Stadnik defined in \cite[Definition 4.4]{stadnikbernsteinsatopoly}. Stadnik works in the context of unit $R[F]$-modules which were introduced by Lyubeznik \cite{lyubFfinite} and Emerton-Kisin \cite{emertonkisinrhunitfcrys}. We briefly recall the relevant notions. A \emph{unit $R[F]$-module} is an $R$-module $\mathcal{M}$ equipped with a structural isomorphsim $\theta \colon F^*\mathcal{M} \xrightarrow{\cong} \mathcal{M}$. A \emph{root} of $\mathcal{M}$ is an $R$-module $M$ together with an injective $R$-linear map $\Phi: M \to F^\ast M$ such that $\colim_e {F^e}^\ast M$ and $\mathcal{M}$ are isomorphic as unit $R[F]$-modules. In particular, if $M$ is a root for $\mathcal{M}$ then $\gamma_\ast M \otimes \omega_{R[t]/R}^{-1}$ is a root for $\gamma_+ \mathcal{M}$, where $\gamma: \Spec R \to \Spec R[t]$ is the graph embedding for some hypersurface $f$ (Proposition \ref{DmodpushforwardlimitCartier} above shows that $\gamma_+$ is the $\mathcal{D}$-module pushforward and \cite[14.3.10, 15.2]{emertonkisinrhunitfcrys} shows that the $\mathcal{D}$-module pushforward coincides with the pushforward on unit $R[F]$-modules). By abuse of notation we will denote the image of the natural map $\gamma_\ast M \otimes \omega^{-1}_{R[t]/R} \to {F^e}^\ast (\gamma_\ast M \otimes \omega_{R[t]/R}^{-1})$ again by $\gamma_\ast M \otimes \omega^{-1}_{R[t]/R}$.

Given a sequence $i_l$ of integers in $\{0, \ldots, p-1\}$ we will refer to $\sum_{l = 0}^{e-1} i_l p^l$ as the \emph{base-$p$ expansion of $(i_1, \ldots, i_e)$}. Varying $e$ we call the number $\lim_{e \to \infty} \sum_{l = 0}^{e-1} i_l p^l/p^e$ \emph{the $p$-weighted limit of the base-$p$ expansion} if it exists.

With this notation Stadnik defines a \emph{$b$-function (for the pair $(\gamma_\ast M \otimes \omega_{R[t]/R}^{-1}, \gamma_+ \mathcal{M})$)} as any polynomial $b(s) \in \mathbb{C}[s]$ with roots in $(0,1]$ that satisfies the following property:

If $\lambda$ is a root of $b(1-s)$ then there exists an integer $n$ such that for all $e \geq 0$ the set \[\{ \lceil \lambda p^e\rceil - a \, \vert\, 0 \leq a \leq p^n\}\] contains the base $p$-expansions of the eigenvalues of the $\theta_{p^l}, l = 0, \ldots, e-1$, on the quotient \[D^e_R[\theta_1, \ldots, \theta_{p^{e-1}}] (\gamma_\ast M\otimes \omega_{R[t]/R}^{-1})/D^e_R[ \theta_1, \ldots, \theta_{p^{e-1}}]t (\gamma_\ast M\otimes \omega_{R[t]/R}^{-1}).\] The set of $b$-functions forms an ideal in $\mathbb{C}[s]$ and we denote its monic generator by $\tilde{b}_{M,f}(s)$.

The main result of \cite{stadnikbernsteinsatopoly} is that $\tilde{b}_{M,f}(s)$ is a non-zero polynomial with rational roots. We will reprove this here using Theorem \ref{MainResult} under the additional assumption that $M \otimes \omega_R$ is $F$-regular and that $f$ is not a zero-divisor on $M$.

\begin{Bem}
Note that $\Omega_{R[t]/R}$ is free of rank $1$ and $dt = d(t-f)$. By definition $\partial_t$ is the differential operator in $\Hom_R(\Omega_{R[t]/R}, R) = R \oplus Der_{R[t]/R}$ given by the dual of $dt$. In particular, we have $\partial_t = \partial_{t-f}$. Hence, applying the automorphism $t \mapsto t + f$ we are precisely in the setting where our hypersurface equation is given by $t = 0$ and the Euler operators are given by $\theta_e = t^{p^e} \partial_t^{[p^{e}]}$, where we use the inclusion $\mathcal{D}_{R[t]/R} \subseteq \mathcal{D}_{R[t]}$. The latter is the setting in which Stadnik works.

Also note that Stadnik considers the quotient \[D^e_R[t, \theta_1, \ldots, \theta_{p^{e-1}}] (\gamma_\ast M \otimes \omega_{R[t]/R}^{-1})/D^e_R[t, \theta_1, \ldots, \theta_{p^{e-1}}]t (\gamma_\ast M\otimes \omega_{R[t]/R}^{-1})\] but since $[t, \theta_i] = \theta_{i-1}$ and $t (\gamma_\ast M\otimes \omega_{R[t]/R}^{-1}) = \gamma_\ast fM \otimes \omega_{R[t]/R}^{-1}$ this quotient coincides with the one we consider.
\end{Bem}

The equivalence of our notion and that of Stadnik is obtained from the equivalence of left and right $\mathcal{D}^e_S$-modules which we recall in a special setting. For an $S$-module $M$ the $e$-th Frobenius pull back\footnote{Note that again we view $F_\ast^e S$ as an $S$-bimodule where the structure on the left is obtained by the ring isomorphism $S \to F_\ast^e S$.} ${F^e}^\ast M = F_\ast^e R \otimes_R M$ is a left $\mathcal{D}^e_S$-module via the action of $\mathcal{D}^e_S$ on $F_\ast^e S$. Tensoring with $\omega_S$ one obtains an isomorphism $\omega_S \otimes {F^e}^\ast M \cong {F^e}^! (\omega_S \otimes M) = \Hom_S(F_\ast^e S, \omega_S \otimes M)$ and the latter naturally carries the structure of a right $\mathcal{D}^e_S$-module via the action on $F_\ast^e S$. Since $\omega_S$ is invertible this induces an equivalence of categories between left and right $\mathcal{D}^e_S$-modules.

Given a smooth ring $S$ over a perfect field $k$ we can choose a set of local coordinates $x_1, \ldots, x_n$. For $i \in \mathbb{N}^n$ write $\partial_{i}^{[m_i]} = \partial_{x_{i_1}}^{[m_1]} \cdots \partial_{x_{i_n}}^{[m_n]}$. Now given any differential operator $P$ we can write it locally as \[\sum_{i \in \mathbb{N}^n} s_i \partial_i^{[m_i]},\] where almost all $i$ are zero. Then we denote the \emph{adjoint operator} \[\sum_{i} (-1)^{\sum_{j=1}^n i_j} \partial_{i}^{[m_i]} s_i\] by $P^t$. Finally, note that one has $(PQ)^t = Q^t P^t$.

\begin{Prop}
\label{EulerLeftRightInterchange}
Let $S = R[t]$ for $R$ smooth over a perfect field $k$ and $M$ an $S$-module. Then for a set of local coordinates $t, x_1, \ldots, x_n$ the right $\mathcal{D}^e_S$-module structure on $\omega_S \otimes {F^e}^\ast M$ is locally given by \[(dt \wedge dx_1 \wedge \ldots \wedge dx_n \otimes m) \cdot P = dt \wedge dx_1 \wedge \ldots \wedge dx_n \otimes P^t m\] and the isomorphism $\omega_S \otimes {F^e}^\ast M \to {F^e}^! (\omega_S \otimes M)$ is $\mathcal{D}^e_S$-linear. In particular, for $1 \otimes v \otimes m \in \omega_{S/R} \otimes_R \omega_R \otimes  {F^e}^\ast M $ one has $(1 \otimes v \otimes m) \cdot \theta_{p^l} = (1 \otimes v \otimes -\vartheta_{p^l} m )$ for any $0 \leq l \leq e-1$.
\end{Prop}
\begin{proof}
First of all, we reduce to the case $M = S$. The $\mathcal{D}^e_S$-module structure on ${F^e}^\ast M = F_\ast^e S \otimes_S M$ and on ${F^e}^! (\omega_S \otimes M) = \Hom_S(F_\ast^e S, \omega_S \otimes M)$ is given by the action of $\mathcal{D}^e_S$ on $F_\ast^e S$. Moreover, the isomorphism $\omega_S \otimes {F^e}^\ast M \to {F^e}^! (\omega_S \otimes M)$ factors as the composition of the canoncial isomorphisms \[\begin{xy}\xymatrix@1{ \omega_S \otimes {F^e}^\ast M \ar[r]& \omega_S \otimes {F^e}^\ast S \otimes M \ar[r]^>>>>>{\Sigma \otimes \id} & {F^e}^! \omega_S \otimes M \ar[r]& {F^e}^! (\omega_S \otimes M) } \end{xy} ,\] where $\Sigma: \omega_S \otimes {F^e}^\ast S \to {F^e}^! \omega_S$ denotes the isomorphism $ds \otimes f \mapsto C^e(ds) \cdot f = [x \mapsto \kappa^e(xf ds)]$.

Now the claim follows from \cite[Proposition 1.1.7 (i), Corollary 1.2.6]{berthelotdmodulesarithmetiqueii} with $\mathcal{M} = \omega_S$ and $\mathcal{E} = F^e_\ast S$: Since $k$ is perfect we have $\mathcal{D}^e_{S/k} = \mathcal{D}^e_{S}$. By \cite[Proposition 2.2.7]{berthelotdmodulesarithmetiquei} the image of $\mathcal{D}^{(e)}_S$ in $\mathcal{D}_S$ corresponds to $\mathcal{D}^{e}_S$ so that Berthelot's results also apply to $\mathcal{D}^e_S$-modules.
\end{proof}

\begin{Ko}
\label{EigenspaceEulerLeftRight}
Let $S = R[t]$ for a regular ring $R$ essentially of finite type over an $F$-finite field and $M$ an $S$-module. Then ${F^e}^\ast M$ admits a non-trivial $(-i_1, \ldots, -i_e)$-eigenspace for the $-\vartheta$ if and only if ${F^e}^! (M \otimes \omega_S)$ admits a non-trivial $(i_1, \ldots, i_e)$-eigenspace for the $\theta$.
\end{Ko}
\begin{proof}
The ``only if''-part is immediate from Proposition \ref{EulerLeftRightInterchange}. Conversely, ${F^e}^! (M \otimes \omega_S) \otimes \omega^{-1}_S$ is canonically isomorphic to ${F^e}^\ast M$ and a similar argument applies in this case.
\end{proof}

\begin{Le}
Let $(M, \Phi)$ be a root of a unit $R[F]$-module $\mathcal{M}$ and $\gamma: \Spec R \to \Spec S = R[t]$ a closed immersion and $(\gamma_\ast M \otimes \omega_{S/R}^{-1}, \Phi \otimes \id)$ the corresponding root for $\gamma_+ \mathcal{M}$.

Then $(\gamma_\ast(\omega_R \otimes M), \tilde{C})$ and $(\omega_R \otimes M, C)$ are naturally Cartier modules and the map $\tilde{C}: \gamma_\ast (\omega_R \otimes M) \to F^! \gamma_\ast \omega_R \otimes M$ is given by the composition of $\gamma_\ast (\omega_R \otimes M) \to \gamma_\ast F^! (\omega_R \otimes M) \to F^!( \gamma_\ast \omega_R \otimes M)$, where the first map is $\gamma_\ast C$ and the second is the composition of maps described in Proposition \ref{DModPushforward}. In particular, if $N$ denotes the image of the natural map $\gamma_\ast M \otimes \omega^{-1}_{R[t]/R} \to {F^e}^\ast (\gamma_\ast M \otimes \omega_{R[t]/R}^{-1})$ then $\omega_{R[t]} \otimes N$ is the image of the natural map $\gamma_\ast( \omega_R \otimes M) \to {F^e}^! (\gamma_\ast \omega_R \otimes M)$.
\end{Le}
\begin{proof}
We shorten $\omega_R \otimes M$ to $M'$. It is easy to see that the Cartier structure given on $\gamma_\ast M'$ is the one induced from $M'$ by \[\begin{xy} \xymatrix@1{F_\ast \gamma_\ast M' \ar[r]^\sim & \gamma_\ast F_\ast M' \ar[r]^{ \gamma_\ast \kappa_{M'}} &\gamma_\ast M'}\end{xy}.\]
Hence, one may reduce the problem to checking that given a Cartier module $(A, \kappa)$ with adjoint $C$ the adjoint of the structural map of $\gamma_\ast A$ is given by the composition of the map described in Proposition \ref{DModPushforward} with $\gamma_\ast C$. This is an easy computation which will be left to the reader.
\end{proof}

\begin{Le}
\label{GraphEmbeddedLeftRightDThetaModules}
Let $R$ be smooth over a perfect field and $f \in R$ a hypersurface. Let $\gamma: \Spec R \to \Spec R[t]$ be the graph embedding along $f$. Given an $R$-module $M$ we have \[\omega_{R[t]} \otimes (\mathcal{D}^e_R[\theta_1, \theta_{p}, \ldots, \theta_{p^{e-1}}] \gamma_\ast M \otimes \omega_{R[t]/R}^{-1})  = (\gamma_\ast \omega_{R} \otimes  M)\mathcal{D}^e_R[\theta_1, \theta_{p}, \ldots, \theta_{p^{e-1}}].\]
\end{Le}
\begin{proof}
According to our established abuse of notation we have to show that for $m$ in the image of $\gamma_\ast M \otimes \omega_{R[t]/R}^{-1} \to {F^e}^\ast \gamma_\ast M \otimes \omega_{R[t]/R}^{-1}$ one has that $\omega \otimes P \cdot m$ for any $P \in D^e_R[\theta_1, \ldots, \theta_{p^{e-1}}]$ is contained in the $\mathcal{D}^e_R[\theta_1, \ldots, \theta_{p^{e-1}}]$-module generated by the image of $\gamma_\ast (\omega_R \otimes M) \to {F^e}^! \gamma_\ast( \omega_R \otimes M)$ and vice versa. We may verify this locally and then it follows from Proposition \ref{EulerLeftRightInterchange} and the fact that $\mathcal{D}^e_R[\theta_1, \ldots, \theta_{p^{e-1}}]$ is closed unter taking adjoints. This is clear for $\mathcal{D}^e_R$ and for $\theta_{p^i}$ one has for the transposed operator $(\theta_{p^i})^t = - \vartheta_{p^i} = -(1 + \theta_{p^i})$ by Lemma \ref{DiffOpFormulae} (i).
\end{proof}

\begin{Ko}
\label{LeftRightIsomQuotients}
In the situation of Lemma \ref{GraphEmbeddedLeftRightDThetaModules} we have an isomorphism
\begin{align*}(&\mathcal{D}^e_R[\theta_1, \ldots, \theta_{p^{e-1}}] \gamma_\ast M \otimes \omega_{R[t]/R}^{-1}/\mathcal{D}^e_R[ \theta_1, \ldots, \theta_{p^{e-1}}]t \gamma_\ast M \otimes \omega_{R[t]/R}^{-1}) \otimes \omega_{R[t]} \cong\\ & (\gamma_\ast\omega_{R} \otimes  M)\mathcal{D}^e_R[\theta_1, \ldots, \theta_{p^{e-1}}]/(\gamma_\ast \omega_R \otimes M)t\mathcal{D}^e_R[\theta_1, \ldots, \theta_{p^{e-1}}].\end{align*}
\end{Ko}
\begin{proof}
Note that $(t \gamma_\ast M) \mathcal{D}^e_R[\theta_1, \ldots, \theta_{p^{e-1}}] = \gamma_\ast (fM \mathcal{D}^e_R) \mathcal{D}^e_R[\theta_1, \ldots, \theta_{p^{e-1}}]$ so that the claim follows from Lemma \ref{GraphEmbeddedLeftRightDThetaModules} and tensoring the obvious short exact sequence with $\omega_S$.
\end{proof}

\begin{Ko}
\label{EulerEigenvaluesonQuotientLeftRight}
The $\mathcal{D}^e_R[\theta_1, \ldots, \theta_{p^{e-1}}]$-module \[D^e_R[\theta_1, \ldots, \theta_{p^{e-1}}] \gamma_\ast M \otimes \omega_{R[t]/R}^{-1} /D^e_R[ \theta_1, \ldots, \theta_{p^{e-1}}]t \gamma_\ast M \otimes \omega_{R[t]/R}^{-1} \] has a non-trivial $\vartheta$-eigenspace with eigenvalue $(-i_1, \ldots, -i_{e})$ if and only if the right $\mathcal{D}^e_R[\theta_1, \ldots, \theta_{p^{e-1}}]$-module \[(\gamma_\ast \omega_R \otimes M)D^e_R[\theta_1, \ldots, \theta_{p^{e-1}}]/(\gamma_\ast \omega_R \otimes M)tD^e_R[\theta_1, \ldots, \theta_{p^{e-1}}]\] has a non-trivial $\theta$-eigenspace with eigenvalue $(i_1, \ldots, i_{e})$.
\end{Ko}
\begin{proof}
Note that the left module embeds into $\mathcal{D}^e_R \gamma_\ast M \otimes \omega_{R[t]/R}^{-1}/\gamma_\ast M \otimes \omega_{R[t]/R}^{-1}$ and similarly for the right module. Then the claim follows from Corollaries \ref{EigenspaceEulerLeftRight} and \ref{LeftRightIsomQuotients}
\end{proof}

\begin{Bem}
One should be able to obtain a similar correspondence between right and left $\mathcal{D}^e_R[\theta_1, \ldots, \theta_{p^{e-1}}]$-modules under the weaker assumption that $R$ is regular, essentially of finite type over an $F$-finite field. However, in this case one cannot appeal to Berthelot's results. Since Stadnik works under the assumptions that $R$ is smooth over a perfect field we did not pursue this further.
\end{Bem}

We now have the necessary ingredients to state and prove the main results of this section.

\begin{Theo}
\label{PWeightedLimitExists}
Let $\gamma: \Spec R \to \Spec R[t]$ be the graph embedding along a hypersurface $f$ and let $(M, \Phi)$ be a root of a unit $R[F]$-module $\mathcal{M}$. Assume that the Cartier module $M \otimes \omega_R$ is $F$-regular\footnote{If $M$ denotes the unique minimal root of $\mathcal{M}$ in the sense of \cite[Definition 2.7]{blicklegammasheaves} then this just means that $M$ is generically simple. That is, any submodule $N$ of $M$ for which $N \to F^\ast M$ factors through $N \to F^\ast N$ which agrees at all generic points of $\Supp M$ with $M$ coincides with $M$.} and that $f$ is not a zero-divisor on $M$.
Then if $\mu_e$ denotes the base-$p$ expansions of the eigenvalues of the $\theta_{p^l}$ operating on the left modules \[\mathcal{D}^e_R[\theta_1, \ldots, \theta_{p^{e-1}}] \gamma_\ast M \otimes \omega_{R[t]/R}^{-1} /\mathcal{D}^e_R[\theta_1, \ldots, \theta_{p^{e-1}}] \gamma_\ast tM \otimes \omega_{R[t]/R}^{-1}\] and $\lambda_e$ denotes the base-$p$ expansions of the eigenvalues of $\theta_{p^l}$ operating on the right modules \[(\gamma_\ast \omega_R \otimes M) \mathcal{D}^e_R[\theta_1, \ldots, \theta_{p^{e-1}}]/ (\gamma_\ast \omega_R \otimes M)t\mathcal{D}^e_R[\theta_1, \ldots, \theta_{p^{e-1}}]\] one has the following relation $\mu_e + \lambda_e = p^e - 1$ and \[\lim_{e \to \infty} \frac{\mu_e}{p^e} = 1 - \lim_{e \to \infty} \frac{\lambda_e}{p^e}.\] In particular, the $p$-weighted limit of the $\mu_e$ exists.
\end{Theo}
\begin{proof}
If $\theta_{p^l}$ operates from the left with eigenvalue $-i_l -1$ then by Corollary \ref{EulerEigenvaluesonQuotientLeftRight} and the relation $\theta_{p^l} +1 = \vartheta_{p^l}$ we have that $\theta_{p^l}$ operates via $i_l$ on the right. So we get $\mu_e = \sum_{l=0}^{e-1} (p-1-i_l)p^l$ as the base-$p$ expansion for the operation of the $\theta_{p^l}$ on the left. Similarly, we have $\lambda_e = \sum_{l=0}^{e-1} i_l p^l$.

By Theorem \ref{MainResult} the $p$-weighted limit over the $\lambda_e$ exists. Moreover, \[\lim_{e \to \infty} \sum_{l=0}^{e-1}\frac{(p-1)p^l}{p^e} = 1\] so that the claim follows.
\end{proof}


We will denote the limit for $e \to \infty$ of the polynomials $b^e_{M,f}(s)$ introduced in Definition \ref{BSPDefinition} by $b_{M,f}(s)$. With this notation we can now compare Stadnik's notion of Bernstein-Sato polynomial to our notion:

\begin{Ko}
Assume the situation of Theorem \ref{PWeightedLimitExists}. Let $\lambda_1, \ldots, \lambda_m$ be $p$-weighted limits of the base-$p$ expansions of the eigenvalues of the $\theta_{p^l}$ acting on \[   (\gamma_\ast \omega_R \otimes M)D^e_R[\theta_1, \ldots, \theta_{p^{e-1}}]/(\gamma_\ast \omega_R \otimes M)tD^e_R[\theta_1, \ldots, \theta_{p^{e-1}}].                                                                                                                    \]
 Then $\tilde{b}_{M,f}(s) = \prod_i (s - \lambda_i)$. In particular, $\tilde{b}_{M,f}(s) = b_{M,f}(s)$
\end{Ko}
\begin{proof}
First of all, we use Theorem \ref{PWeightedLimitExists} to ensure that the $p$-weighted limits actually do exist. Recall from the discussion before Theorem \ref{MainResult} that for all $e \geq 1$ the base-$p$ expansions of the eigenvalues of $\theta$ are given by $\lceil \lambda p^e\rceil - 1$. This implies that $b(s)$ is a Bernstein-Sato polynomial.

In the other direction we have to show that $b(s)$ is minimal in the sense that we may not omit any of the $\lambda_i$. Assume that we have omitted $\lambda_m$ and for some $n \geq 0$ the set $\{\lceil \lambda_m p^e \rceil - a \, \vert \, 0 \leq a \leq p^n \}$ is contained in the set $\{\lceil \lambda_i p^e \rceil - a \, \vert \, 0 \leq a \leq p^n, 1 \leq i \leq m-1 \}$. In particular the $\lceil \lambda_m p^e \rceil -1$ are all contained in this set. Since all parameters except $e$ of this set are finite we may assume that $\lceil \lambda_m p^e \rceil -1 = \lceil \lambda_i p^e \rceil - a$ for some fixed $a, i$ and infinitely many $e$. Dividing by $p^e$ and passing to the limit $e \to \infty$ yields $\lambda_i = \lambda_m$ -- a contradiction.
\end{proof}

\bibliography{bibliothek}
\bibliographystyle{amsplain}
\end{document}